\documentclass{amsart}
\usepackage{amsfonts,amssymb,amscd,amsmath,enumerate,verbatim,calc}
\usepackage[all]{xy}

\newcommand{\CM}{Cohen-Macaulay}

\newcommand{\wrt}{with respect to}

\newcommand{\n}{\mathfrak{n} }
\newcommand{\m}{\mathfrak{m} }

\newcommand{\g}{\mathfrak{gr} }

\newcommand{\ZZ}{\mathbb{Z} }
\newcommand{\FF}{\mathcal{F}}
\newcommand{\GG}{\mathbb{G}}
\newcommand{\EE}{\mathbb{E}}

\newcommand{\C}{\mathbf{C} }

\newcommand{\rt}{\rightarrow}

\newcommand{\image}{\operatorname{image}}

\newcommand{\Ass}{\operatorname{Ass}}

\newcommand{\height}{\operatorname{height}}

\newcommand{\Supp}{\operatorname{Supp}}

\newcommand{\injdim}{\operatorname{injdim}}

\newcommand{\Hom}{\operatorname{Hom}}
\newcommand{\Ext}{\operatorname{Ext}}

\theoremstyle{plain}

\newtheorem{theorem}{Theorem}[section]
\newtheorem{corollary}[theorem]{Corollary}
\newtheorem{lemma}[theorem]{Lemma}
\newtheorem{proposition}[theorem]{Proposition}

\theoremstyle{definition}
\newtheorem{definition}[theorem]{Definition}
\newtheorem{s}[theorem]{}
\newtheorem{remark}[theorem]{Remark}

\theoremstyle{remark}

\begin{document}

\title[invariant rings]{Local cohomology modules  of invariant rings}
 \author{Tony J. Puthenpurakal}
\date{\today}
\address{Department of Mathematics, Indian Institute of Technology Bombay, Powai, Mumbai 400 076, India}
\email{tputhen@math.iitb.ac.in}


\begin{abstract}
Let $K$ be a field  and let $R$ be a regular domain containing $K$. 
Let $G$ be a finite subgroup of the group of  automorphisms of $R$. We assume that $|G|$ is invertible in $K$. Let $R^G$ be the ring of invariants of $G$. Let $I$ be an ideal in $R^G$.  Fix $i \geq 0$. 
 If $R^G$ is Gorenstein then,
\begin{enumerate}[ (i)]
\item
$\injdim_{R^G} H^i_I(R^G) \leq \dim \Supp H^i_I(R^G).$
\item
 $H^j_{\m}(H^i_I(R^G))$ is injective, where $\m$ is any maximal ideal of $R^G$.
 \item
 $\mu_j(P, H^i_I(R^G)) = \mu_j(P^\prime, H^i_{IR}(R))$ where $P^\prime$ is any prime in $R$ lying above $P$.
\end{enumerate}
We also prove that if $P$ is a prime ideal in $R^G$ with $R^G_P$ \emph{not Gorenstein}
 then either the bass numbers $\mu_j(P, H^i_I(R^G)) $ is zero for all $j$ or there exists $c$ such that $\mu_j(P, H^i_I(R^G)) = 0 $ for $j < c$ and $\mu_j(P, H^i_I(R^G)) > 0$ for all $j \geq c$.
\end{abstract}

\maketitle
\section{Introduction}
Throughout this paper $R$ is a commutative Noetherian ring. If $M$ is an $R$-module and if $I$ is an ideal in $R$, we denote by $H^i_I(M)$ the  $i^{th}$ local cohomology module with respect to $I$.
 
In a remarkable paper, \cite{HuSh}, Huneke and Sharp proved that if $R$ is a regular ring containing a field of characteristic $p > 0$, and $I$ is an ideal in $R$ then the local cohomology modules of $R$ \wrt \ $I$ have the following properties:
\begin{enumerate}[\rm(i)]
\item
$H^j_{\m}(H^i_I(R))$ is injective, where $\m$ is any maximal ideal of $R$.
\item
$\injdim_R H^i_I(R) \leq \dim \Supp H^i_I(R)$.
\item
The set of associated primes of $H^i_I(R)$ is finite.
\item
All the Bass numbers of $H^i_I(R)$ are finite.
\end{enumerate}
Here $\injdim_R H^i_I(R)$ denotes the injective dimension of $H^i_I(R)$. Also $\Supp M = \{ P \mid  M_P \neq 0 \ \text{and $P$ is a prime in $R$}\}$ is the support of an $R$-module $M$. The $j^{th}$ Bass number of an $R$-module $M$ with respect to a prime ideal $P$ is defined as $\mu_j(P,M) = \dim_{k(P)} \Ext^j_{R_P}(k(P), M_P)$ where $k(P)$ is the residue field of $R_P$.

 In another remarkable paper, for regular rings in characteristic zero, Lyubeznik was able to establish the above properties  
for a considerably larger class of functors than just the local cohomology modules, see  \cite{Lyu-1}. In particular for ideals $I_1,\ldots,I_n$ in $R$ and $T(R) = H^{i_1}_{I_1}(H^{i_2}_{I_2}(\cdots H^{i_n}_{I_n}(R) \cdots))$ then $T(R)$ satisfies the following properties:
  \begin{enumerate}[\rm(i)]
\item
$H^j_{\m}(T(R))$ is injective, where $\m$ is any maximal ideal of $R$.
\item
$\injdim_R T(R) \leq \dim \Supp T(R)$.
\item
For every maximal ideal $\m$, the number of  associated primes of $T(R)$ contained in $\m$
 is finite.
\item
All the Bass numbers of $T(R)$ are finite.
\end{enumerate}
 This, in turn, raised the question of whether the results (i)–-(iv) of Huneke and Sharp (in characteristic $p>0$) could be extended to this larger class of functors. In \cite{Lyu-2}, Lyubeznik proves it.
 
 For singular rings analogus results are  in general false. Hartshorne gave an example of  a singular ring $R$, an ideal $I$ and a maximal ideal $\m$ of $R$  such that $\mu_0(\m, H^2_I(R))$ is infinite, see \cite[Sect. 3]{Hart}. Singh gave the first example of a singular ring $R$ having an ideal $I$ such that $\Ass_R H^i_I(R)$ is infinite, see \cite{AS}. In this example the ring $R$ did not contain a field. Later Katzman, see \cite{K}, gave an example of an affine algebra $R$   over a field  (and also a local ring containing a field) having an ideal $I$ such that $\Ass_R H^i_I(R)$ is infinite. Later Singh and Swanson gave similar examples of a ring having only rational singularities, see \cite{ASIS}. 
 
 In a nice paper N\'{u}\~{n}ez-Betancourt, proved that
  if $S \rt R$ is a homomorphism of Noetherian rings that splits, then for every ideal $I$ in $S$ and every non-negative integer $i$, if $\Ass_R H^i_{IR}(R)$ is finite then $\Ass_S H^i_I(S)$ is finite. In addition, if $R$ is Cohen-Macaulay and finitely generated as an $S$-module and all Bass numbers of the $R$-modules $H^i_{IR}(R)$ are finite, then all Bass numbers of the $S$-modules $H^i_I(S)$ are finite. 
 
 A case when the above result holds is when $R$ is a regular domain containing a field $K$ and $G$ is a finite group acting on $R$ with $|G|$ invertible in $K$  and $S = R^G$. Our result is that in this case much more is true.
 
 \begin{theorem}\label{main}
Let $K$ be a field  and let $R$ be a regular domain containing $K$. 
Let $G$ be a finite subgroup of the group of  automorphisms of $R$ with $|G|$  invertible in $K$. Let $R^G$ be the ring of invariants of $G$. Let $I_1,I_2,\cdots, I_r$ be ideals in $R^G$. Set $T(R^G) = H^{i_1}_{I_1}(H^{i_2}_{I_2}(\cdots H^{i_r}_{I_r}(R^G) \cdots)$  for some $i_1,\cdots,i_r \geq 0$. 
\begin{enumerate}[\rm (i)]
\item
If $R^G$ is Gorenstein then
\begin{enumerate}[\rm (a)]
\item
$\injdim_{R^G} T(R^G) \leq \dim \Supp T(R^G).$
\item
Let $P$ be a prime ideal in $R^G$. Then
$\mu_j(P, T(R^G)) = \mu_j(P^\prime, T(R))$ where $T(R) = H^{i_1}_{I_1R}(H^{i_2}_{I_2R}(\cdots H^{i_r}_{I_rR}(R) \cdots)$  and $P^\prime$ is any prime in $R$ lying above $P$.
\end{enumerate}
\item
Let $P$ be a prime ideal of $R^G$ with $R^G_P$ \emph{not Gorenstein}. Then for all $j \geq 0$, either the Bass numbers
$\mu_{j}(P, T(R^G)) = 0$ for all $j$ or  there exists $c$ such that
 $\mu_j(P, T(R^G)) = 0 $ for $j < c$ and $\mu_j(P, T(R^G)) > 0$ for all $j \geq c$.
\end{enumerate}
\end{theorem} 
The main example where our Theorem applies is when $R = K[X_1,\ldots,X_n]$ or $R = K[[X_1,\ldots,X_n]]$ and $G$ is a finite subgroup of $GL_n(K)$ acting linearly on $R$, with $|G|$ invertible in $K$. 
In this case we should note that,   by a result due to K. Watanabe, $R^G$ is Gorenstein if $G \subseteq SL_n(K)$; see \cite{W}. My motivation was to understand local cohomology modules in this case. However to prove the result for this special case I had to prove the general result.

It is perhaps of some interest to explicitly compute local cohomology modules via computer algebra software packages. I prove a finiteness result which I hope will help in this direction. Let $R = K[X_1,\ldots,X_n]$ where $K$ is a field of characteristic zero and $G$ is a finite subgroup of $GL_n(K)$ acting linearly on $R$. Let $D(R)$ be the ring of $K$-linear differential operators on $R$. It is well-known that $D(R)$ is isomorphic to $A_n(K)$, the $n^{th}$-Weyl algebra over $K$. It is possible to extend the action of $G$ on $D(R)$; see   \ref{def-action}. Let $D(R)^G$ be the ring of invariants.  There are algorithms to compute $D(R)^G$, see \cite{T}. We prove
\begin{theorem}[with hypotheses as above]\label{main-2}
Let $I$ be an ideal in $R^G$. Then for all $i \geq 0$, $H^i_I(R^G)$ is a $D(R)^G$-module of finite length.
\end{theorem}
  
The main technical tool in this paper is the skew group ring of $R$ with respect to $G$; we denote it by $R*G$.
 We prove that certain local cohomology  modules become naturally a module over the skew group ring and this
 has an impact to its structure.
 
 We now describe in brief the contents of this paper.  In section two we discuss some preliminary results on skew group rings. In section three  we prove our results regarding skew group rings and local cohomology.  In section four we discuss injective resolution of a module over the skew group ring and discuss its application to local cohomology. We then apply these results in the next section to prove a lemma regarding $H^j_P(-)_P$ which we apply in the next two sections. In section six we assume that $R^G$ is Gorenstein and prove the first part of Theorem \ref{main}. In the next section we consider the case when $R^G$ is not Gorenstein and prove the final part of Theorem \ref{main}.  Finally in section eight we prove \ref{main-2}.

\section{skew group rings}
In this section $A$ is a ring (not necessarily commutative) and $G$ is a finite subgroup of $Aut(A)$; the group of automorphisms of $A$. We assume that $|G|$ is invertible in $A$. In this section we describe some of the basic properties of the skew group ring $A*G$ that we will need. Most of the results here are perhaps already known. However absence of a good reference forces me to include all proofs.

\s Recall that 
$$A*G = \{ \sum_{\sigma \in G} a_\sigma \sigma \mid a_{\sigma} \in A \ \text{for all} \ \sigma \},$$
with multiplication defined as
\[
(a_{\sigma}\sigma)(a_\tau \tau) = a_{\sigma}\sigma(a_{\tau})\sigma \tau.
\]  

\begin{remark}
An $A*G$ module $M$ is precisely an $A$-module on which $G$ acts such that for all $\sigma \in G$,
\[
\sigma(am) = \sigma(a)\sigma(m) \quad \text{for all} \ a \in A \ \text{and} \ m \in M.
\] 
\end{remark}

\begin{definition}
Let $M$ be an $A*G$-module. Then
\[
M^G = \{ m \in M \mid \sigma(m) = m \ \text{for all} \ \sigma \in G \}.
\]
\end{definition}
In particular set $A^G$ to be the ring of invariants of $G$. Clearly $M^G$ is an $A^G$-module.
It can also be easily checked that if $u \colon M \rt N$ is $A*G$-linear then $u(M^G) \subseteq N^G$ and the restriction map $\widetilde{u} \colon M^G \rt N^G$ is $A^G$-linear. Thus we have a functor $(-)^G \colon Mod(A*G) \rt Mod(A^G)$. It can be verified that $(-)^G = \Hom_{A*G}(A,-)$; so in particular it is left exact.

\s For any $A*G$ module $M$ we have a Reynolds operator 
\begin{align*}
\rho^M \colon M &\rt M^G \\
              m &\rt \frac{1}{|G|}\sum_{\sigma \in G} \sigma m.
\end{align*}
Clearly $\rho^M(m) = m$ for all $m \in M^G$. Also $\rho^M$ is $A^G$-linear and it splits the inclusion $M^G \rt M$. 

We now show that taking invariants is an exact functor.
\begin{lemma}\label{fixed-exact}
Let $0 \rt M_1 \xrightarrow{u_1} M_2 \xrightarrow{u_2} M_3 \rt 0$ be a short exact sequence of $A*G$-modules. 
Then the induced sequence
\[
0 \rt M_1^G \xrightarrow{\widetilde{u_1}}M_2^G \xrightarrow{\widetilde{u_2}} M_3^G \rt 0
\]
is also exact.
\end{lemma}
\begin{proof}
We have already observed the fixed point functor is left exact. Thus it suffices to prove that $\widetilde{u_2}$ is surjective. Let $\xi \in M_3^G$. As $u_2$ is surjective there exists $t \in M_2$ with $u_2(t) = \xi$. For any $\sigma \in G$ notice
\[
u_2(\sigma t) = \sigma u_2(t) = \sigma \xi = \xi.
\]
It follows that
\[
u_2(\rho^{M_2}(t)) = \xi.
\]
It follows that $\widetilde{u_2}$ is surjective.
\end{proof} 
The following result is interesting.
\begin{lemma}\label{descent-simple}
Let $M$ be a simple $A*G$-module. Then either $M^G = 0$ or it is a simple $A^G$-module.
\end{lemma}
\begin{proof}
Suppose $M^G \neq 0$. Let $N \neq 0$ be an $A^G$-submodule of $M^G$. Let $t \in N$ be non-zero. Let $\xi \in M^G$ be an arbitrary   non-zero element. 

As $M$ is a simple $A*G$-module we have $M = A*G t$. So $\xi = \alpha t$ for some $\alpha \in A*G$. Say $\alpha = \sum_{\sigma \in G}a_\sigma \sigma$. So
\[
\xi = \alpha t = \sum_{\sigma \in G}a_\sigma \sigma t = (\sum_{\sigma \in G}a_\sigma)t.
\]
The last equality holds since $\sigma t = t$ for all $\sigma \in G$. Set $d = \sum_{\sigma \in G}a_\sigma$. So
$\xi = d t$. 
It follows that for $\sigma \in G$
\[
\xi = \sigma(\xi) = \sigma(d)\sigma(t) = \sigma(d) t.
\]
It follows that
\[
\xi = \frac{1}{|G|}\sum_{\sigma \in G} \sigma(\xi) = \frac{1}{|G|}\sum_{\sigma \in G} (\sigma(d) t) = \left( \frac{1}{|G|}\sum_{\sigma \in G} \sigma(d)\right)t.
\]
So $\xi = \rho^A(d)t$ and thus $\xi \in N$. Thus $N = M^G$. It follows that $M^G$ is a simple $A^G$-module.
\end{proof}
An easy consequence of the previous Lemma is the following result.
\begin{corollary}\label{descent-Artin}
Let $M$ be an $A*G$-module of finite length. Then $M^G$ has finite length as an $A^G$-module.
\end{corollary}
\begin{proof}
$M$ has finite length as an $A*G$-module. So there is a filtration 
 $$ 0 = M_0 \subseteq M_1 \subseteq M_2 \subseteq \cdots M_{n-1} \subseteq M_n = M,$$
 such that $M_i/M_{i-1}$ is a simple $A*G$-module for $i= 1,\cdots,n$.
 
 Notice $0 = M_0^G \subseteq M_1^G \subseteq M_2^G \subseteq \cdots M_{n-1}^G \subseteq M_n^G = M^G$ is an filtration of $M^G$ as an $A^G$-module. The exact sequence
 \[
 0 \rt M_{i-1} \rt M_i \rt M_i/M_{i-1} \rt 0,
 \]
 yields 
 \[
 0 \rt M_{i-1}^G \rt M_i^G \rt (M_i/M_{i-1})^G \rt 0.
 \]
 By Lemma \ref{descent-simple} $(M_i/M_{i-1})^G$ is either zero or is simple as an $A^G$-module. It follows that $M^G$ has finite   length as an $A^G$-module.
\end{proof}

The next result shows that the fixed point operator commutes with talking homology. More precisely we have the following result.
\begin{theorem}\label{commute-hom}
Let $\C \colon \cdots \rt M_n \xrightarrow{u_n} M_{n-1} \xrightarrow{u_{n-1}}M_{n-2} \rt \cdots $ be a complex of 
$A*G$-modules. Consider
$$\C^G \colon \cdots \rt  M_n^G   \xrightarrow{\widetilde{u_n}} M_{n-1}^G \xrightarrow{\widetilde{u_{n-1}}} M_{n-2}^G \rt \cdots . $$
Then
\begin{enumerate}[\rm(1)]
\item
$\C^G$ is a complex of $A^G$-modules.
\item
For each $n \in \ZZ$ there is an $A^G$-linear map $\psi_n \colon H_n(\C^G) \rt H_n(\C)$.
\item
For all $n$ the map $\psi_n$ is injective.
\item
 For all $n \in \ZZ$ we have 
$\image \psi_n = H_n(\C)^G.$
\item
 For all $n \in \ZZ$ we have $H_n(\C^G) \cong H_n(\C)^G$ as $A^G$-modules
\end{enumerate}
\end{theorem}
\begin{proof}
The assertion (1) is clear. Furthermore (5) follows from (3) and (4). 

(2) Fix $n \in \ZZ$. Let $z \in Z_n(\C^G)$. So $\widetilde{u_n}(z) = 0$. Thus $u_n(z) = 0$. It follows that $z \in Z_n(\C)$. If $b \in B_n(\C^G)$ then let $\widetilde{u_{n+1}}(t) = b$. Then $u_{n+1}(t) = b$ and so $b \in B_n(\C)$. Thus we have a well-defined map
\begin{align*}
\psi_n \colon H_n(\C^G) &\rt H_n(\C) \\
            z + B_n(\C^G) &\rt z + B_n(\C).
\end{align*}
Clearly $\psi_n$ is $A^G$-linear.

(3) Fix $n \in \ZZ$. Let $\xi \in H_n(\C^G)$ be such that $\psi_n(\xi) = 0$. Say $\xi = z + B_n(\C^G)$.
Then $\psi_n(\xi) = z + B_n(\C)$. So $z \in B_n(\C)$. Thus there exists $t \in M_{n+1}$ such that $u_{n+1}(t) = z$.
Let $\sigma \in G$. Notice
\[
u_{n+1}(\sigma t) = \sigma u_{n+1}(t) = \sigma z = z.
\]
So $u_{n+1}(\sigma t) = z$ for all $\sigma \in G$. It follows that
\[
u_{n+1}(d) = z \quad \text{where}  \ d = \rho^{M_{n+1}}(t).
\]
As $d \in M_{n+1}^G$ we have $\widetilde{u_{n+1}}(d) = z$. It follows that $z \in B_n(\C^G)$. So $\xi = 0$. Thus $\psi$ is an injective map.

(4) It is clear that $\image(\psi_n) \subseteq  H_n(\C)^G.$  Suppose $\xi \in H_n(\C)^G$. Say $\xi = z + B_n(\C)$.
Let $\sigma \in G$. As $\sigma \xi = \xi$ we have $\sigma z = z + v_{\sigma} $, where $v_\sigma \in B_n(\C)$.
It follows that
\[
y = \rho^{M_n}(z) = z + v, \quad \text{where} \ v = \frac{1}{|G|}\sum_{\sigma \in G}v_\sigma \in B_n(\C)
\]
Clearly $y \in Z_n(\C^G)$. Notice
\[
\psi_n( y + B_n(\C^G)) =  y + B_n(\C) = z + B_n(\C) = \xi.  
\]
The result follows.
\end{proof}
\section{Skew group rings and local cohomology}
Let $A$ be a commutative Noetherian ring and let $G \subseteq Aut(A)$ be a finite group with $|G|$ invertible in $A$. Let $A^G$ be the ring of invaritants of $G$. Let $A*G$ be the skew group ring of $A$ \wrt \ $G$. In this section we show that certain local cohomology modules over $A$ has a natural $A*G$-module structure. We then investigate some of its properties.

\begin{lemma}\label{local}
Let $M$  be an $A*G$-module and let $S \subseteq A^G$ be a multiplicatively closed set. Then
\begin{enumerate}[\rm (1)]
\item
$S^{-1}M$ is an $A*G$-module.
\item
$S^{-1}M^G$ can be naturally identified with a subset of $S^{-1}M$ and with this identification we have
$(S^{-1}M)^G = S^{-1}M^G$.
\end{enumerate}
\end{lemma}
\begin{proof}
(1) We first define a $G$-action on $S^{-1}M$. Let $\sigma \in G$ and let $\xi \in S^{-1}M$. If $\xi = m/s$ then we define $\sigma(\xi) = \sigma(m)/s$. We first show that this is well-defined. If $\xi = m_1/s_1 = m_2/s_2$ then there 
exists $s_3 \in S$ with $s_3s_2m_1 = s_3s_1m_2$. As $S \subseteq A^G$ we have $s_3s_2\sigma(m_1) = s_3s_1\sigma(m_2)$. So $\sigma(m_1)/s_1 = \sigma(m_2)/s_2$ in $S^{-1}M$. Thus the action of $G$ on $S^{-1}M$ is well-defined. It is easy to see it is a $G$-action on $S^{-1}M$.

Let $a \in A$ and let $\xi \in S^{-1}M$. Say $\xi = m/s$. Then
\[
\sigma(a \xi) = \sigma(a m)/s = \sigma(a)\sigma(m)/s = \sigma(a)\sigma(\xi).
\]
It follows that $S^{-1}M$ is an $A*G$-module.

(2) We have $M^G \subseteq M$ and this inclusion is $A^G$-linear. It follows that $S^{-1}M^G \subseteq S^{-1}M$.
Clearly we have $S^{-1}M^G \subseteq (S^{-1}M)^G$. 

Let $\xi \in (S^{-1}M)^G$. Say $\xi = m/s$. Then for every $\sigma \in G$ we have 
$\sigma(m)/s = m/s$. It follows that there exists $s_{\sigma }^\prime \in S$ such that $s_{\sigma}^\prime s \sigma(m) = s_{\sigma}^\prime s m$. Set $s_{\sigma} = ss_{\sigma}^\prime \in S$. Notice $\sigma(s_{\sigma}m) = s_{\sigma}m$.
Put
\[
\theta = \prod_{\sigma \in G}s_\sigma \in S.
\]
Notice $\sigma(\theta m) = \theta m$ for every $\sigma \in G$. So $\theta m \in M^G$. It follows that
\[
\xi = m/s = (\theta m)/(\theta s) \in S^{-1}M^G.
\]
\end{proof}
The main result in this section is the following:
\begin{theorem}\label{loc-invar}
Let $M$ be an $A*G$-module and let $I$ be an ideal in $A^G$. Then
\begin{enumerate}[\rm (1)]
\item
$H^i_{IA}(M)$ is an $A*G$-module for every $i \geq 0$.
\item
$H^i_{IA}(M)^G \cong H^i_I(M^G)$ for all $i \geq 0$.
\end{enumerate}
\end{theorem}
\begin{proof}
(1) Let $I = (f_1,\ldots,f_s)$.
Consider the \textit{(modified)} \v{C}ech complex
\[
\C \colon 0\rt M \rt \bigoplus_{i=1}^{s}M_{f_i} \rt \cdots \rt M_{f_1\cdots f_s} \rt 0.
\]
\textit{Claim:} $\C$ is a complex of $A*G$-modules.

By Lemma \ref{local} each module in $\C$ is an $A*G$-module. So we have to prove that each differential in $\C$ is $A*G$-linear. To prove this it suffices to prove that if $f,g \in A^G$ then the natural map $\eta \colon M_f \rt M_{fg}$ is $A*G$-linear. Clearly $\eta$ is $A$-linear. Let $\xi \in M_f$. Say $\xi = m/f^i$. Then
$\sigma \xi = \sigma(m)/f^i$. So
$\eta(\sigma \xi) = g^i\sigma(m)/f^ig^i$. Notice
\[
\sigma \eta(\xi) = \sigma(g^im/f^ig^i) = g^i\sigma(m)/f^ig^i = \eta(\sigma \xi).
\]
Thus $\eta$ is $A*G$-linear. So $\C$ is a complex of $A*G$-modules. It follows that $H^i_{IA}(M)$ is an $A*G$-module for all $i \geq 0$.

(2). Note the complex $\C^G$ as defined in Theorem \ref{commute-hom} is the \v{C}ech complex on $M^G$. By Theorem \ref{commute-hom} it follows that $H^i_{IA}(M)^G \cong H^i_I(M^G)$ for all $i \geq 0$.
\end{proof}
As a consequence of the above Theorem we get the following:
\begin{corollary}\label{cor-commute-loc}
Let $I, I_1,\ldots, I_r$ be ideals in $A^G$. Then
\begin{enumerate}[\rm (1)]
\item
$H^i_{IA}(A)$ is an $A*G$-module for all $i \geq 0$. Furthermore
$$H^i_{IA}(A)^G \cong H^i_I(A^G).$$
\item
For all $i_j \geq 0$, where $j = 1,\ldots,r$, $H^{i_1}_{I_1A}(H^{i_2}_{I_2A}(\cdots H^{i_r}_{I_rA}(A) \cdots)$
is an $A*G$-module. Furthermore
$$ H^{i_1}_{I_1A}(H^{i_2}_{I_2A}(\cdots H^{i_r}_{I_rA}(A) \cdots)^G \cong H^{i_1}_{I_1}(H^{i_2}_{I_2}(\cdots H^{i_r}_{I_r}(A^G) \cdots).$$
\end{enumerate}
\end{corollary}
\begin{proof}
(1) This follows from Theorem \ref{loc-invar} since $A$ is an $A*G$-module.

(2) We prove this result by induction on $r$. For $r = 1$ this is just part (1).
Assume the result for $r-1$ where $r \geq 2$.  Fix $i_2,\cdots, i_r \geq 0$. Set 
$$M = H^{i_2}_{I_2A}(\cdots H^{i_r}_{I_rA}(A) \cdots).$$
By induction hypotheses $M$ is an $A*G$-module and 
\[
M^G \cong H^{i_2}_{I_2}(\cdots H^{i_r}_{I_r}(A^G) \cdots).
\]
By Theorem \ref{loc-invar} it follows that for all $i_1 \geq 0$, $H^{i_1}_{I_1A}(M)$ is an $A*G$-module and 
\[
H^{i_1}_{I_1A}(M)^G \cong H^{i_1}_{I_1}(M^G).
\]
The result follows.
\end{proof}
We also need the following Lemma.
\begin{lemma}\label{aut}
Let $A$ be a domain and let $G \subseteq Aut(A)$ be a finite group. Let $S$ be a multiplicatively closed subset of $A^G$. Let $G$ act on $S^{-1}A$ as given in Lemma \ref{local}. Then  the natural map $G \rt Aut(S^{-1}A)$ is injective.  
\end{lemma}
\begin{proof}
Suppose $\sigma \in G$ such that the map $\sigma\colon S^{-1}A \rt S^{-1}A$ is the identity.
So $\sigma(\xi) = \xi$ for every $\xi \in S^{-1}A$. Let $a \in A$. Then $\sigma(a/1) = a/1.$ This gives
$\sigma(a)/1 = a/1$. As $A$ is a domain we have $\sigma(a) = a$. It follows that $\sigma$ is the identity.
\end{proof}

\section{equivariant injective resolution}
In this section $A$ is a normal domain with quotient field $L$. Also $G$ is a finite subgroup of the  group of automorphisms of $A$. We assume that $|G|$ is invertible in $A$. Let $A^G$ be the ring of invariants of $G$. Then $A^G$ is normal, see \cite[6.4.1]{BH}. Let $F$ be the quotient field of $A^G$. Note that $G$ acts on $L$ and $L^G = F$. Thus $L$ is a Galois extension of $F$ and the Galois group is $G$. Let $\m$ be a maximal ideal of $A^G$. Let $\n_1,\ldots,\n_r$ be \emph{all} the maximal ideals of $A$ lying above $\m$. By \cite[9.3]{Mat} for $i,j$; there exists $\sigma_i^j \in G$ such that $\sigma_i^j(\n_i) = \n_j$. Our main result in this section is:
\begin{theorem}\label{invar-co}
(with hypotheses as above). Let $M$ be an $A*G$-module. Then for $n \geq 0$,
\begin{enumerate}[\rm (1)]
\item
 the local cohomology module $H^n_{\m A}(M)$ is an $A*G$-module.
\item
As $A$-modules,
 $$H^n_{\m A}(M) = \bigoplus_{i = 1}^{r}\Gamma_{\n_i}\left(H^n_{\m A}(M) \right).$$
 \item
 For $i = 1,\ldots, r$;
 $$ \Gamma_{\n_i}\left(H^n_{\m A}(M) \right)  \cong H^n_{\n_i}(M).$$
 \item
 If $\sigma_i^j(\n_i) = \n_j$ then
 \[
 \sigma_i^j\left(\Gamma_{\n_i}\left(H^n_{\m A}(M) \right) \right) = \Gamma_{\n_j}\left(H^n_{\m A}(M)\right).
 \]
\end{enumerate}
\end{theorem}
We need a few preliminaries before we prove this result.
\begin{remark}
Clearly $A*G$ is free as a left $A$-module. Note that for any $a \in A$ and 
$\sigma \in G$ we have $\sigma a = \sigma(a) \sigma $. Also $\sigma \colon A \rt A$ is an automorphism. It follows that $A*G$ is also free as a right $A$-module. 
\end{remark}
A significant consequence of the above remark is the following:
\begin{lemma}\label{invar-inj}
Let $E$ be an injective $A*G$-module. Then $E$ is injective as an $A$-module.
\end{lemma}
\begin{proof}
Notice
\begin{align*}
\Hom_A(-, E) &= \Hom_A(-, \Hom_{A*G}(A*G,E)), \\
             &= \Hom_{A*G}(A*G\otimes_A -, E).
\end{align*}
As $A*G$ is free as a right $A$-module we have that $A*G\otimes_A-$ is an exact functor from $Mod(A)$ to $Mod(A*G)$. Also by hypothesis $E$ is an injective $A*G$-module. It follows that $\Hom_A(-, E)$ is an exact functor. So $E$ is injective as an $A$-module. 
\end{proof}
We now give 
\begin{proof}[Proof of Theorem \ref{invar-co}]
The assertion (1) follows from Theorem \ref{loc-invar}.

Let $\EE$ be an injective resolution of $M$ as an $A*G$-module. By \ref{invar-inj} it is also an injective resolution of $M$ as an $A$-module. Notice as $\sqrt{\m A} = \n_1\cdots\n_r$ and as they are co-maximal we have
\[
\Gamma_{\m A}(\EE^n) = \bigoplus_{i=1}^{r} \Gamma_{\n_i}(\EE^n).
\]
Let $d^n$ be the $n^{th}$ differential of the complex $\Gamma_{\m A}(\EE)$. Write $d^n = (u_{i,n}^j)$ where
$u_{i,n}^j \colon \Gamma_{\n_i}(\EE^n) \rt \Gamma_{\n_j}(\EE^{n+1})$ is $A$-linear.

\textit{Claim 1}: $u_{i,n}^j = 0$ for $i \neq j$. \\
\textit{Proof of Claim 1:} Let $a \in \Gamma_{\n_i}(\EE^n)$. Then $u_{i,n}^j(a) \in \Gamma_{\n_j}(\EE^{n+1})$.
So there exists $t_1,t_2$ such that $\n_i^{t_1}a = 0$ and $\n_j^{t_2} u_{i,n}^j(a) = 0$. As $i \neq j$ we have
$\n_i^{t_1} + \n_j^{t_2} = A$. So $1 = c_i + c_j$ where $c_i \in \n_{i}^{t_1}$ and $c_j \in \n_{j}^{t_2}$.
So $a = c_ia + c_ja = c_ja$. Therefore
\[
u_{i,n}^j(a) = c_ju_{i,n}^j(a) = 0.
\]
As a consequence of Claim 1 we get that the module of $n$-cocycles
\[
Z^n(\Gamma_{\m A}(\EE)) = \bigoplus_{i=1}^{r}Z^n(\Gamma_{\n_i}(\EE)),
\]
and the module of $n$-coboundaries
\[
B^n(\Gamma_{\m A}(\EE)) = \bigoplus_{i=1}^{r}B^n(\Gamma_{\n_i}(\EE)).
\]
It follows that
\[
H^n_{\m A}(M) = \bigoplus_{i=1}^{r}H^n_{\n_i}(M).
\]

(2),(3): It follows from above that $\Gamma_{\n_i}(H^n_{\m A}(M)) = H^n_{\n_i}(M)$.

(4). Let $\sigma_{i}^j(\n_i) = \n_j$. We first assert that $\sigma_i^j(\Gamma_{\n_i}(\EE^n)) = \Gamma_{\n_j}(\EE^n)$.
Let $\xi \in  \Gamma_{\n_i}(\EE)$. So $\n_i^s \xi = 0$ for some $s \geq 1$. Also note that $\sigma_{i}^j(\n_i^s) = \n_j^s$. Let $b \in \n_j^s$. There exists $a \in \n_i^s$ with $\sigma_i^j(a) = b$.
Notice
\[
b\sigma_i^j(\xi) = \sigma_i^j(a)\sigma_i^j(\xi) = \sigma_i^j(a \xi) = 0.
\]
Thus $\sigma_i^j(\xi) \in \Gamma_{\n_j}(\EE^n)$. Thus $\sigma_i^j(\Gamma_{\n_i}(\EE^n)) \subseteq \Gamma_{\n_j}(\EE^n)$. By considering $(\sigma_i^j)^{-1}$ we get
$(\sigma_i^j)^{-1}(\Gamma_{\n_j}(\EE^n)) \subseteq \Gamma_{\n_i}(\EE^n)$. Thus  $\sigma_i^j(\Gamma_{\n_i}(\EE^n)) = \Gamma_{\n_j}(\EE^n)$.

As $d^n$ is $A*G$-linear, the following diagram is commutative:
\[
  \xymatrix
{
  & \Gamma_{\n_i}(\EE^n) 
    \ar@{->}[d]^{\sigma_i^j}
\ar@{->}[r]^{u^i_{i,n}}
 & \Gamma_{\n_i}(\EE^{n+1})  
    \ar@{->}[d]^{\sigma_i^j}
 \\
  & \Gamma_{\n_j}(\EE^n) 
\ar@{->}[r]^{u^i_{i,n}}
 & \Gamma_{\n_j}(\EE^{n+1})     
 }
\]

In a similar way we can prove that
\[
\sigma_i^j(Z^n(\Gamma_{\n_i}(\EE))) = Z^n(\Gamma_{\n_j}(\EE)) \quad \text{and} \ \sigma_i^j(B^n(\Gamma_{\n_i}(\EE))) = B^n(\Gamma_{\n_j}(\EE)).
\]
It follows that 
\[
\sigma_i^j(H^n_{\n_i}(M)) = H^n_{\n_j}(M).
\]

\end{proof}
\section{A Crucial Lemma}
In this section we prove  a lemma which will play a crucial part in the proof  of Theorem \ref{main}. 
\begin{lemma}\label{bass-prep}
Let $K$ be a field  and let $R$ be a regular domain containing $K$. 
Let $G$ be a finite subgroup of the group of  automorphisms of $R$. We assume that $|G|$ is invertible in $K$. Let $R^G$ be the ring of invariants of $G$. 
Let $I_1,I_2,\cdots, I_r$ be ideals in $R^G$. Set $T(R^G) = H^{i_1}_{I_1}(H^{i_2}_{I_2}(\cdots H^{i_r}_{I_r}(R^G) \cdots)$  for some $i_1,\cdots,i_r \geq 0$.
Let $P$ be a prime ideal of $R^G$ with $\height P = g$. Then for all $j \geq 0$, 
\[
\left(H^j_P(T(R^G))\right)_P = H^g_{PR^G_P}(R^G_P)^{s} \quad \text{for some finite} \ s \geq 0.
\]
\end{lemma}

We also need the following result. This is well-known, however I do not have a reference, so I prove it.
\begin{proposition}\label{tor-loc}
Let $A$ be a Noetherian  ring and let $\m$ be a maximal ideal in $A$. Let $M$ be an $\m$-torsion $A$-module($M$ need not be finitely generated). Then $M = M_{\m}$. 
\end{proposition}
\begin{proof}
It suffices to prove that for every $s \in A \setminus \m$ the map $\mu_s \colon M \rt M$ given by multiplication by $s$ is an isomorphism. We first prove $\mu_s$ is surjective. Let $t \in M$. As $M$ is $\m$-torsion there exists $n \geq 1$ such that $\m^nt=0$. Notice $\m^n + As = A$. Let $1 = \xi + as$ where $\xi \in \m^n$ and $a \in A$. So $t =\xi t + as t = as t$. Thus $\mu_s(at) = t$. Thus $\mu_s$ is surjective.

Next we prove that $\mu_s$ is injective. Say $\mu_s(t) = 0$. So $st =0$. Say $\m^nt= 0$. As before $1 = \xi + as$ where $\xi \in \m^n$ and $a \in A$. So $t =\xi t + as t = 0$. Thus $\mu_s$ is injective.  Therefore $\mu_s$ is an isomorphism for every $s \in A \setminus \m$. It follows that $M = M_\m$.
\end{proof}

\begin{proof}[Proof of Lemma \ref{bass-prep}]
Let $L$ be quotient field of $R$ and let  $F$ be quotient field of $R^G$. Note $L$ is a Galois extension of $F$ with Galois group $G$. It is also clear  $R^G$ is normal and that the integral closure of $R^G$ in $L$ is $R$. Set 
$$T(R) = H^{i_1}_{I_1R}(H^{i_2}_{I_2R}(\cdots H^{i_r}_{I_rR}(R) \cdots).$$

Set $A = R_P$. Then by \ref{aut}, $G$ acts via automorphisms on $A$ and $A^G = R^G_P$; see \ref{local} and 4.4.
Note $A^G$ is normal and the integral closure of $A^G$ in $L$ is $A$, see \cite[5.12]{AM}.
Note $PA^G$ is the unique maximal ideal
of $A^G$. Let $P_1,\ldots,P_r$ be maximal ideals in $A$ lying above $P$. It can be easily verified that $\height P_l = g$ for $l = 1,\ldots,r$. By \cite[9.3]{Mat}, for $k,l$; there exists $\sigma_k^l \in G$ such that $\sigma_k^l(P_k) = P_l$.

 Let $M = T(A) = H^{i_1}_{I_1A}(H^{i_2}_{I_2A}(\cdots H^{i_r}_{I_rA}(A) \cdots)$. Then $M$ is an $A*G$-module. So by   Theorem \ref{invar-co} we have 
 \[
 H^j_{PA}(M) = \bigoplus_{l = 1}^{r}H^j_{P_l}(M) \quad \text{and} \ \sigma_k^l(H^j_{P_k}(M)) = H^j_{P_l}(M) \ \text{for all} \ l,k.
 \]
 We should note that $T(R)_P = M$.
 Notice by \ref{tor-loc} 
 \begin{align*}
 H^j_{P_l}(M)    &= H^j_{P_l}(M)_{P_l} \\
                &= H^j_{P_lA_{P_l}}(M_{P_l}) \\
               &= H^j_{P_lR_{P_l}}(T(R)_{P_l}) \\
               &= \left(H^l_{P_l}( H^{i_1}_{I_1R}(H^{i_2}_{I_2R}(\cdots H^{i_r}_{I_rR}(R) \cdots)  )\right)_{P_l}.
 \end{align*}
 It follows from Lyubeznik results, \cite[3.4]{Lyu-1} in characteristic zero and \cite[1.5, 2.14]{Lyu-2} in characteristic $p >0$,  that
 \[
  \left(H^l_{P_l}( H^{i_1}_{I_1R}(H^{i_2}_{I_2R}(\cdots H^{i_r}_{I_rR}(R) \cdots)  )\right)_{P_l} = E_R(R/P_l)^{t_l}_{P_l} \quad \text{for some  finite}\ t_l \geq 0.
 \]
 Also notice by \ref{tor-loc},
 \[
 H^g_{P_lA}(A) = H^g_{P_lA_{P_l}}(A_{P_l}) =  E_R(R/P_l)_{P_l}.
 \]
 So we have 
 \[
 H^j_{PA}(M) = \bigoplus_{l = 1}^{r}H^g_{P_lA}(A)^{t_l}.
 \]
 We also have that
 \[
 \sigma_k^l(H^g_{P_lA}(A)^{t_l}) = H^g_{P_kA}(A)^{t_k} \ \text{for all} \ l,k.
 \]
 It follows that $t_1 = t_2 = \cdots = t_r$. Put $s = t_1$.
Then
\[
H^j_P(M) =  \left(\bigoplus_{l = 1}^{r}H^g_{P_lA}(A)\right)^{s} \cong H^g_{PA}(A)^{s};
\] 
 as $A*G$-modules. Taking invariants we have
 \[
 H^j_{PA^G}(M^G) \cong H^p_{PA^G}(A^G)^{s} \cong H^g_{PR^G_P}(R^G_P)^{s}.
 \]
 Notice that 
 \[
 H^j_{PA^G}(M^G) \cong \left(H^j_P(T(R^G))\right)_P.
 \]
 The result follows. 
\end{proof}
\begin{remark}\label{lying above}
(with hypotheses as above)
Set $A = R_P$.
If $N = T(R)_P$  and if $P_1,\ldots, P_r$ are the prime ideals in $R$ lying above $P$ then we showed that 
$$H^j_{PA}(N) \cong \bigoplus_{l = 1}^{r}H^g_{P_l}(A)^s,$$
and
$$H^j_P(T(R^G)_P) \cong H^g_{PR^G_P}(R^G_P)^s.$$
The point to note that the same constant $s$ appears in both the above equations.
\end{remark}
\section{The case when $R^G$ is Gorenstein}
In this section we prove the first part of Theorem 1.1.
\begin{theorem}\label{Gor}
Let $K$ be a field  and let $R$ be a regular domain containing $K$. 
Let $G$ be a finite subgroup of the group of  automorphisms of $R$. We assume that $|G|$ is invertible in $K$. Let $R^G$ be the ring of invariants of $G$.   We further assume that $R^G$ is a Gorenstein ring. 
Let $I_1,I_2,\cdots, I_r$ be ideals in $R^G$. Set $T(R^G) = H^{i_1}_{I_1}(H^{i_2}_{I_2}(\cdots H^{i_r}_{I_r}(R^G) \cdots))$  and $T(R) =   H^{i_1}_{I_1R}(H^{i_2}_{I_2R}(\cdots H^{i_r}_{I_rR}(R) \cdots))$ for some $i_1,\cdots,i_r \geq 0$. Then
\begin{enumerate}[\rm (i)]
\item
$\injdim_{R^G} T(R^G) \leq \dim \Supp T(R^G).$
\item
Let $P$ be a prime ideal in $R^G$. Then
$\mu_j(P, T(R^G)) = \mu_j(P^\prime, T(R))$ where  $P^\prime$ is any prime in $R$ lying above $P$.
\end{enumerate}
\end{theorem}
We will need the following Lemma from \cite[1.4]{Lyu-1}.
\begin{lemma}\label{lyu-lemma}
Let $A$ be a Noetherian ring and let $M$ be an $A$-module ($M$ need not be finitely generated).
Let $P$ be a prime ideal in $A$. If $(H^j_P(M))_P$ is injective for all $j \geq 0$ then
$\mu_j(P,M) = \mu_0(P,H^j_P(M))$.
\end{lemma}
We now give
\begin{proof}[Proof of Theorem \ref{Gor}]
Put $M = T(R^G)$. Let $P$ be a prime ideal of $R^G$ of height $g$.

(i) By
Lemma \ref{bass-prep} we have
\[
H^j_P(M)_P = H^g_{PR^G_P}(R^G_P)^{s} \quad \text{for some finite} \ s \geq 0.
\]
As $R^G$ is Gorenstein we have that $R^G_P$ is Gorenstein local. So
$$ H^g_{PR^G_P}(R^G_P) \cong E_{R^G_P}(R^G_P/PR^G_P) \cong E_{R^G}(R^G/P),$$
is an injective $R^G$-module. Thus by \ref{lyu-lemma} we have that
\[
\mu_j(P,M) = \mu_0(P,H^j_P(M)).
\]
By Grothendieck vanishing theorem $H^j_P(M) = 0$ for $j > \dim \Supp M$. Thus 
$\mu_j(P,M) = 0$ for all $j > \dim \Supp M$ and for any prime $P$ of $R^G$. So if $\EE$ is a minimal injective resolution of $M$ we have $\EE^j = 0$ for $j > \dim \Supp M$. Thus
\[
\injdim M \leq  \dim \Supp M.
\]
(ii) We localize at $P$. We have 
\[
\mu_j(P,M) = \mu_0(P,H^j_P(M)).
\]
The result now follows from \ref{lying above}.
\end{proof}
\section{The case when $R^G$ is not Gorenstein}
In this section we prove the second part of Theorem \ref{main}. We restate it here for the convenience of the reader.
\begin{theorem}\label{main-inj}
Let $K$ be a field  and let $R$ be a regular domain containing $K$. 
Let $G$ be a finite subgroup of the group of  automorphisms of $R$ with $|G|$  invertible in $K$. Let $R^G$ be the ring of invariants of $G$. Let $I_1,I_2,\cdots, I_r$ be ideals in $R^G$. Set $T(R^G) = H^{i_1}_{I_1}(H^{i_2}_{I_2}(\cdots H^{i_r}_{I_r}(R^G) \cdots)$  for some $i_1,\cdots,i_r \geq 0$. 

Let $P$ be a prime ideal of $R^G$ with $R^G_P$ \emph{not Gorenstein}. Then for all $j \geq 0$, either the Bass numbers
$\mu_{j}(P, T(R^G)) = 0$ for all $j$ or  there exists $c$ such that $\mu_j(P, T(R^G)) = 0 $ for $j < c$ and $\mu_j(P, T(R^G)) > 0$ for all $j \geq c$.
\end{theorem} 
\begin{proof}
After localizing it suffices to prove the result for maximal ideals, see \ref{local} and \ref{aut}. Let $\m$ be a maximal ideal in $R^G$ and let $d = \height \m$. Let $M = T(R^G)$, $E = E_{R^G}(R^G/\m)$ and let $l = R^G/\m$.

Let $\GG$ be a minimal injective resolution of $M$. Write
\[
\GG^j = \widetilde{\GG^j} \oplus E^{r_j} \quad \text{with} \ \m \notin \Ass(\widetilde{\GG^j}).
\]
Thus $\mu_j(\m , M) = r_j$ for $j \geq 0$. We know that $r_j$ is finite for all $j \geq 0$. Suppose there exists $j$ such that $r_j > 0$. Let
\[
c = \min\{ j \mid r_j > 0 \}.
\]
We  prove that $r_j > 0$ for all $j \geq c$.

Set $\EE = \Gamma_\m(\GG)$. Note $\EE^j = E^{r_j}$ for all $j \geq 0$. Furthermore 
by Lemma \ref{bass-prep} and Proposition 7.3 we have
\[
H^j(\EE) = H^j_\m(T(R^G)) \cong H^d_\m(R^G)^{s_{ij}} \quad \text{for some finite} \ s_{ij} \geq 0.
\] 

 Let $S$ be the completion of $R^G$ at $\m$. Also notice that 
 $E = E_{S}(S/\m S)$.  Let $(-)^{\vee} $ be the Matlis dual functor of $S$.
 
 Let $Z^j, B^j$ be the module of $j$-co-cycles and $j$-co-boundaries of the complex $\EE$. We prove the following assertion by induction on $j \geq c$.
 \begin{enumerate}
 \item
 $Z^j \neq 0$.
 \item
 $\injdim Z^j  = \infty$.
 \item
 $(Z^j)^\vee$ is a non-free  maximal \CM \ $S$-module.
 \item
 $B^{j+1} \neq 0$.
 \item
 $\injdim B^{j+1}  = \infty$.
 \item
 $(B^{j+1})^\vee$ is a non-free  maximal \CM \ $S$-module.
 \end{enumerate}
 It is convenient to prove all the assertions together for $j \geq c$. Note that   (1) will imply our assertion.

We prove the result for $j = c$.
Notice that as $\GG$ is a minimal injective resolution of $M$ we have that the map
\[
\Hom_{R_G}(l, \GG^j) \rt \Hom_{R_G}(l, \GG^{j+1}) \quad \text{is zero}.
\]
Notice $\Hom_{R_G}(l, \widetilde{\GG^j}) = 0$ for all $j \geq 0$. It follows that for all $j \geq 0$ the map
\[
\Hom_{R_G}(l, E^{r_j}) \rt \Hom_{R_G}(l, E^{r_{j+1}}) \quad \text{is zero}.
\]
It follows that $Z^c \neq 0$. Also $Z^c = H^d_\m(R^G)^s$ for some finite $s > 0$. Notice $H^d_\m(R^G)
 = H^d_\m(S)$, use \ref{tor-loc} and \cite[3.5.4(d)]{BH}. As $S$ is not Gorenstein it follows that $H^d_\m(S)$ is not injective. Furthermore  $H^d_\m(S)^\vee = \omega $ the canonical module of $S$ is  a non-free maximal \CM \ $B$-module. Thus $Z^c$ has infinite injective dimension and $(Z^c)^\vee = \omega^s$ is a non-free maximal \CM \ $S$-module.
 
 We have an exact sequence $0 \rt Z^c \rt E^{r_c} \rt B^{c+1} \rt 0$. As $\injdim Z^c = \infty$ it follows that $B^{c+1} \neq 0$ and has infinite injective dimension. 
 By taking Matlis duals we have an exact sequence
 \[
 0 \rt (B^{c+1})^\vee \rt S^{r_c} \rt (Z^c)^{\vee} \rt 0.
 \]
 It follows that  $(B^{c+1})^\vee$ is a non-free  maximal \CM \ $S$-module.
 
 Now assume that the result holds for $j = n$. We prove it for $j = n+1$. We have an exact sequence
 \[
 0 \rt B^{n+1} \rt Z^{n+1} \rt H^{n+1}(\EE) \rt 0.
 \]
 So $Z^{n+1} \neq 0$.
If $H^{n+1}(\EE) = 0$ then clearly $Z^{n+1}$ satisfies properties  (2) and (3). 
If $H^{n+1}(\EE) \neq 0$ then it is equal to $H^d_\m(S)^{s_n}$ for some finite $s_n > 0$. Taking Matlis-duals we obtain 
\[
0 \rt \omega^{s_n} \rt (Z^{n+1})^\vee \rt (B^{n+1})^\vee \rt 0.
\]
For any maximal \CM \ $S$-module $N$ we have $\Ext^1_S(N,\omega) = 0$. It follows that
 \[
 (Z^{n+1})^\vee \cong \omega^{s_n} \oplus (B^{n+1})^\vee.
 \]
 It follows that $(Z^{n+1})^\vee$ is a non-free maximal \CM \ $S$-module. Also by taking duals again we get that 
 \[
 Z^{n+1} \cong B^{n+1} \oplus H^d_\m(S)^{s_n};
 \] 
 has infinite injective dimension. 
 
 We have an exact sequence $0 \rt Z^{n+1} \rt E^{r_{n+1}} \rt B^{n+2} \rt 0$. As $\injdim Z^{n+1} = \infty$ it follows that $B^{n+2} \neq 0$ and has infinite injective dimension. 
 By taking Matlis duals we have an exact sequence
 \[
 0 \rt (B^{n+2})^\vee \rt S^{r_{n+1}} \rt (Z^{n+1})^{\vee} \rt 0.
 \]
 It follows that  $(B^{n+2})^\vee$ is a non-free  maximal \CM \ $S$-module.
\end{proof}

\section{Ring of invariants of differential operators and local cohomology}
Let $K$ be a field of characteristic zero and let $R = K[X_1,\ldots,X_n]$. Let $G$ be a finite subgroup of $GL_n(K)$ acting linearly on $R$. Let $R^G$ be the ring of invariants of $R$. Let $D(R)$ be the ring of $K$-linear differential operators on $R$. Note $D(R) \cong A_n(K)$ the $n^{th}$-Weyl algebra over $K$. We recall a natural action of $G$ on $D(R)$, cf. \cite[Section 1]{T} and then consider the ring of invariants $D(R)^G$.

\s We first recall the construction of $D(R)$ as a subring of $ S = \Hom_K(R,R)$. The composition of two elements $P, Q$ of $S$ will be denoted as $P\cdot Q$. The commutator of
$P$ and $Q$ is the element 
\[
[P,Q] = P \cdot Q - Q \cdot P.
\]
We have natural inclusion $\eta \colon R \rt S$ where $\eta(r) \colon R \rt R$ is multiplication by $r$.

Set $D_0(R) = R$ viewed as a subring of $S$. For $i \geq 1$ set
\[
D_i(R) = \{ P \in S \mid [P,r] \in D_{i-1}(R) \}.
\]
Elements of $D_i(R)$ are said to be differential operators on $R$ of degree $\leq i$.
Notice $D_{i+1}(R) \supseteq D_i(R)$ for all $i \geq 0$. Set
\[
D(R) = \bigcup_{i \geq 0}D_i(R).
\]
This is the ring of $K$-linear differential operators on $R$. It can be shown that 
$D(R) \cong A_n(K)$.  Set $D(R)_{-1} = 0$. Note that the graded ring
\[
\g(D(R)) = \bigoplus_{i\geq 0}D(R)_i/D(R)_{i-1} = R[\overline{\partial_1},\cdots\overline{\partial_n}],
\]
is isomorphic to the polynomial ring in $n$-variables over $R$.
\s \label{def-action} We define action of $G$ on $D(R)$ as follows. Let $\theta \in D(R)_i$. Let $g \in G$. Define
\begin{align*}
g\theta \colon R &\rt R \\
           r &\rt g\cdot\theta(g^{-1}r). \\
\end{align*}
It can be verified that $g\theta \in D(R)_i$. Thus we have an action of $G$ on $D(R)$. It is easily verified that $G \hookrightarrow Aut(R)$.

\begin{remark}\label{simple-ex}
\begin{enumerate}
\item
Let $s \in R$ and let
 $\mu_s \colon R \rt R$ be  the multiplication by $s$. Then $g \mu_s = \mu_{gs}$.
 \item
 Let $g \in G \subset GL_n(K)$ be given by matrix $T_g$ then it can be verified that
 \[
 g\begin{bmatrix}\partial_1 \\ \cdots \\ \partial_n  \end{bmatrix} = (T_g^{-1})^t\begin{bmatrix}\partial_1 \\ \cdots \\ \partial_n  \end{bmatrix}.
 \]
 Here $(-)^t$ indicates the transpose of the matrix. Notice $g\partial_i$ is a derivation for all $g \in G$ and for all $i$.
\end{enumerate}
\end{remark}
\s Let $D(R)^G$ be the ring of invariants of $G$. For $i \geq 0$ let $\FF_i = D(R)_i \cap D(R)^G$. By \ref{simple-ex}.(1) we have that $\FF_0 = R^G$. The graded ring
$\g(\FF)$ is a subring of $\g(D(R))$. By \cite[Theorem 1]{T} there is a natural $G$-action on $\g(D(R))$ with $\g(D(R))^G = \g(\FF)$. It follows that 
\begin{enumerate}
\item
$D(R)^G$ is Noetherian, since $\g(\FF)$ is Noetherian.
\item
$\g(D(R))$ is a finitely generated $\g(\FF)$-module. So $D(R)$ is finitely generated as a $D(R)^G$-module.
\end{enumerate}
Altough we will not use these facts, I felt that it is important enough to be pointed out.

\s Let  $D(R)*G$ be the skew group ring of $D(R)$ \wrt \ $G$. Let $ D(R)^G$ be the ring of invariants. Clearly $D(R)$ is a $D(R)*G$-module.

\begin{proposition}\label{basic}
$R$ is an $D(R)*G$-module.
\end{proposition}
\begin{proof}
Clearly $R$ is an $D(R)$-module and that $G$ acts on $R$. We want to prove that
for any $\theta \in D(R)$, $g \in G$  and $r \in R$ we have $g(\theta \cdot r) = g(\theta)\cdot g(r)$.

Notice $\theta \cdot r = \theta(r)$ for any $\theta \in D(R)$. Thus
\begin{align*}
g(\theta \cdot r)&= g(\theta(r)) \ \text{while}; \\
g(\theta)\cdot g(r) &= g(\theta)[g(r)], \\
                     &= g[\theta(g^{-1}gr)], \\
                     &= g(\theta(r)).
\end{align*} 
Thus $R$ is an $D(R)*G$-module.
\end{proof}

\s Notice the ring $R*G$ is clearly a subring of $D(R)*G$. Thus $D(R)*G$-modules are naturally $R*G$-modules. Let $M$ be a $D(R)*G$-module. Let $f \in R^G$. We have a natural $G$-action on $M_f$, see \ref{local}. It is also well-known that $M_f$ is a $D(R)$-module. We have
\begin{proposition}\label{prep}[with hypotheses as above]
$M_f$ is an $D(R)*G$-module.
\end{proposition}
\begin{proof}
We have to prove that for all $\sigma \in G$, $\theta \in D(R)$ and $\xi \in M_f$ that 
\begin{equation}\label{eqn}
\sigma(\theta \xi) = \sigma(\theta)\sigma(\xi).
\end{equation}
As $D(R)$ is generated by $R$ and the derivations it suffices to prove \ref{eqn} when $\theta \in R$ or is a derivation. When $\theta \in R$ then clearly \ref{eqn} holds, see \ref{local}.

So now assume that $\theta$ is a derivation. By \ref{simple-ex}(2), $\sigma(\theta)$ is also a derivation.
Let $\xi = m/f^i$. We first compute
\begin{align*}
\sigma(\theta \xi) &=  \sigma(\theta(m/f^i)), \\
 &= \sigma\left( \frac{\theta m}{f^i} - \frac{i\theta(f)m}{f^{i+1}} \right), \\
 &= \frac{\sigma(\theta m)}{f^i}  -  \frac{i\sigma(\theta(f))\sigma(m)}{f^{i+1}}.
\end{align*}
Next we compute $\sigma(\theta)\sigma(\xi)$. As $\sigma(\theta)$ is a derivation we get
\begin{align*}
\sigma(\theta)\sigma(\xi) &= \sigma(\theta)\left( \frac{\sigma(m)}{f^i} \right),\\
 &=\frac{\sigma(\theta) \sigma(m)}{f^i} - \frac{i (\sigma(\theta)(f))\sigma(m)}{f^{i+1}}.
\end{align*}
Notice
\begin{enumerate}[\rm(1)]
\item
$ \sigma(\theta m) = \sigma(\theta)\sigma(m).$
\item
$\sigma(\theta)(f) = \sigma(\theta(\sigma^{-1}f))  = \sigma(\theta(f)); \text{since} \ f \in R^G. $
\end{enumerate}
Thus $\sigma(\theta \xi) = \sigma(\theta)\sigma(\xi)$. It follows that $M_f$ is a $D(R)*G$-module.
\end{proof}

We now prove the following analogue of Theorem \ref{loc-invar}.
\begin{theorem}\label{loc-invar-2}
Let $M$ be a $D(R)*G$-module. Let $I$ be an ideal in $R^G$. Then
\begin{enumerate}[\rm (1)]
\item
$H^i_{IR}(M)$ is a $D(R)*G$-module for all $i \geq 0$.
\item
$H^i_{IR}(M)^G \cong H^i_I(M^G)$ for all $i \geq 0$.
\end{enumerate}
\end{theorem}
\begin{proof}
(1)  Let $I = (f_1,\ldots,f_s)$. Consider the \v{C}ech complex
\[
\C \colon 0\rt M \rt \bigoplus_{i=1}^{s}M_{f_i} \rt \cdots \rt M_{f_1\cdots f_s} \rt 0
\]
\textit{Claim:} $\C$ is a complex of $D(R)*G$-modules.

By Lemma \ref{prep} each module in $\C$ is an $D(R)*G$-module. So we have to prove that each differential in $\C$ is $D(R)*G$-linear. To prove this it suffices to prove that if $f,g \in R^G$ then the natural map $\eta \colon M_f \rt M_{fg}$ is $D(R)*G$-linear. Clearly $\eta$ is $D(R)$-linear. By an argument similar to in \ref{loc-invar} we get that
$\eta(\sigma\xi) = \sigma(\eta(\xi))$ for any $\sigma \in G$ and $\xi \in M_f$. Thus $\eta$ is $D(R)*G$-linear. 
It follows that $\C$ is a complex of $D(R)*G$-modules. Therefore $H^i_{IR}(M)$ is a $D(R)*G$-module for all $i \geq 0$.

(2). Note the complex $\C^G$ as defined in Theorem \ref{commute-hom} is the \v{C}ech  complex on $M^G$. By Theorem \ref{commute-hom} it follows that $H^i_{IA}(M)^G \cong H^i_I(M^G)$ for all $i \geq 0$.
\end{proof}
As an easy consequence we obtain a proof of  Theorem \ref{main}. We  restate it here for the convenience of the reader.
\begin{theorem}\label{main-dr}(with hypotheses as above)
Let $I$ be an ideal in $R^G$. Then for all $i \geq 0$, $H^i_I(R^G)$ is a finite length $D(R^G)$-module.
\end{theorem}
\begin{proof}
We apply Theorem \ref{loc-invar-2} to the case $M = R$.
Note $H^i_{IR}(R)$ is a holonomic $D(R)$-module. So it has finite length as a $D(R)$-module. It follows that $H^i_{IR}(R)$ has finite length as a $D(R)*G$-module. So by \ref{descent-Artin} we get that $H^i_I(R^G) = H^i_{IR}(R)^G$ has finite length as a $D(R)^G$-module.
\end{proof}

\end{document}